\documentclass[11pt,psamsfonts]{article}
\usepackage{amsmath,amssymb,latexsym}
\usepackage{amscd}
\usepackage[mathscr]{eucal}

\newcommand\qed{{\hspace*{\fill}Q.E.D.\vskip12pt plus 1pt}}

\newcommand\Pic[1]{\hbox{\rm Pic(}#1\hbox{\rm )}}

\def\Coker{\operatorname{Coker}}

\def\h{\hat}
\def\rank{\operatorname{rank}}
\def\inv{\operatorname{inv}}

\def\ara{\operatorname{ara}}

\def\lim{\operatorname{lim}}
\def\Lef{\operatorname{Lef}}
\def\Leff{\operatorname{Leff}}

\def\Tors{\operatorname{Tors}}
\def\Alb{\operatorname{Alb}}
\def\an{\operatorname{an}}

\def\Spec{\operatorname{Spec}}

\def\red{\operatorname{red}}
\def\mod{\operatorname{mod}}

\def\Proj{\operatorname{Proj}}

\def\codim{\operatorname{codim}}

\def\Pic{\operatorname{Pic}}

\def\id{\operatorname{id}}
\def\rank{\operatorname{rank}}

\def\Coker{\operatorname{Coker}}

\def\Im{\operatorname{Im}}

\def\lcm{\operatorname{lcm}}
\newcommand{\car}{\mathrm{char}}

\newcommand\proof{\noindent{\em Proof.}\ \ }

\newtheorem{theorem}{Theorem}[section]
\newtheorem{lemma}[theorem]{Lemma}
\newtheorem{corollary}[theorem]{Corollary}

\newtheorem{prop}[theorem]{Proposition}

\newtheorem{question}[theorem]{Question}

\newtheorem{definition}[theorem]{Definition}
\newtheorem{rem}[theorem]{Remark}
\newtheorem{rems}[theorem]{Remarks}
\newtheorem{pargrph}[theorem]{}
\newtheorem{examp}[theorem]{Example}
\newtheorem{examps}[theorem]{Examples}
\newtheorem{MM}[theorem]{ }
\newtheorem{res}[theorem]{Remarks}

\makeatletter
\ifnum\@ptsize=0  \fi
\ifnum\@ptsize=2  \fi
\renewcommand{\qed}{\hfill $\square$}

\newenvironment{rem*}{\begin{rem}\em}{\end{rem}}
\newenvironment{rems*}{\begin{res}\em}{\end{res}}
\newenvironment{example*}{\begin{examp}\em}{\end{examp}}
\newenvironment{definition*}{\begin{definition}\em}{\end{definition}}
\newenvironment{question*}{\begin{question}\em}{\end{question}}
\newenvironment{MM*}{\begin{MM}\em}{\end{MM}}

\newenvironment{prgrph*}[1]{\indent\begin{pargrph}{\bf #1.}\em\
}{\end{pargrph}}

\begin{document}

\title{{Grothendieck-Lefschetz Theory, Set-Theoretic Complete Intersections and  Rational Normal Scrolls}
}
\author{Lucian B\u adescu and Giuseppe Valla}
\date{}

\maketitle

\begin{abstract} 
\noindent Using the Grothendieck-Lefschetz theory (see \cite{[SGA2]}) we prove a criterion 
to deduce that certain subvarieties of $\mathbb P^n$ of dimension $\geq 2$ are not set-theoretic complete intersections  (see Theorem 1 of the Introduction). As applications we give a number of relevant examples. In the last part of the paper we prove that the arithmetic rank of a rational normal $d$-dimensional scroll $S_{n_1,\ldots,n_d}$ in $\mathbb P^N$ is $N-2$, by producing an explicit set of $N-2$ homogeneous equations which define these scrolls set-theoretically (see Theorem 2 of the Introduction).
\end{abstract}

\section*{Introduction} 

Let us start by recalling  the following definition.

\medskip

\noindent{\bf Definition.}  Let $Y$ be a closed irreducible subvariety of the projective space $\mathbb P^n$, and denote by $\mathscr I_+(Y)$ the homogeneous prime ideal generated by all the homogeneous polynomials in $k[T_0,T_1,\ldots, T_n]$ (in $n+1$ variables) that vanish at each point of $Y$. If $f_1,\ldots,f_r\in k[T_0,T_1,\ldots, T_n]$ are  homogeneous polynomials, denote also by $\mathscr V_+(f_1,\ldots,f_r)$ the locus of points of $\mathbb P^n$ where $f_1,\ldots,f_r$ vanish. The {\em arithmetic rank} of $Y$ in $\mathbb P^n$,  denoted by  $\ara(Y)$,  is the minimal number of homogeneous equations needed to define $Y$  set-theoretically in $\mathbb P^n$, i.e. $\ara(Y)$ is the minimal natural number $r$ for which  there exist $r$ homogeneous polynomials $f_1,\ldots,f_r\in k[T_0,T_1,\ldots,T_n]$  such that $\mathscr V_+(f_1,\ldots,f_r)=Y$. By Nullstellensatz, $\ara(Y)$ is the minimal natural number $r$ for which there exist $r$ homogeneous polynomials $f_1,\ldots,f_r\in k[T_0,T_1,\ldots,T_n]$ such that $\mathscr I_+(Y)=\root\of {(f_1,\ldots,f_r)}$.   Clearly, $\ara(Y)\geq\codim_{\mathbb P^n}(Y)$.
If  $\ara(Y)=\codim_{\mathbb P^n}(Y)$, we say that $Y$ is a {\em set-theoretic complete intersection} in $\mathbb P^n$.

\medskip

This paper has two main parts. In the first part we show how the Grothendieck-Lefschetz theory (see \cite{[SGA2]}) can be used to provide necessary conditions for a given subvariety $Y$ of dimension $d\geq 2$ of the projective space $\mathbb P^n$ (over an algebraically closed field of arbitrary characteristic)
to be a set-theoretic complete intersection in $\mathbb P^n$. We shall illustrate this through a number of relevant examples. In the second part of the paper we show that the arithmetic rank of any rational normal scroll $S$ of dimension $\geq 2$ in $\mathbb P^N$ is $N-2$, by exhibiting an explicit minimal set of $N-2$ defining equations for $S$.

The paper is organized as follows. In Section 1 we recall some basic results from Grothendieck-Lefschetz theory that are going to be used in Section 2. We also recall two Lefschetz theorems (for singular cohomology and for \'etale cohomology) that will be used in Section 4.

In Section 2, using Grothendieck-Lefschetz theory (together with some basic results from the theory of Picard schemes, see Grothendieck \cite{[FGA]}), we prove the following result (see also Theorem 
\ref{main} below):

\medskip

\noindent{\bf Theorem 1} {\em Let $Y$ be a closed irreducible
subvariety of $\mathbb P^n$ of dimension $\geq 2$ over an algebraically closed field $k$
of characteristic $p\geq 0$. 
\begin{enumerate}\item[\em i)] If $Y$ is a set-theoretic complete intersection in $\mathbb P^n$ then $Y$
is algebraically simply connected, i.e. there are no non-trivial connected \'etale covers of $Y$.
\item[\em ii)]
Assume  that $p=0$ and  $Y$ normal. If  $H^1(\mathscr O_Y)\neq 0$, then $Y$ is not a set-theoretic complete intersection in $\mathbb P^n$.
\item[\em iii)] Assume  that $p>0$ and  $Y$ is normal. If $H^1(\mathscr O_Y)\neq 0$ and the Picard scheme $\underline{\Pic}_Y^0$ of $Y$ is reduced, then $Y$ is not a set-theoretic complete intersection in $\mathbb P^n$. $($If for example $H^2(\mathscr O_Y)=0$, then  $\underline{\Pic}_Y^0$ is always reduced, see $\cite{[FGA]}$, \'Expos\'e $236$, Proposition
$2.10$, $ii))$.
\item[\em iv)]  Assume that $Y$ is a set-theoretic complete intersection in $\mathbb P^n$. Then the restriction map $\alpha\colon\Pic(\mathbb P^n)\to\Pic(Y)$ is injective and $\Coker(\alpha)$ is torsion-free if $p=0$, and has no $s$-torsion for every integer $s>0$ which is prime to $p$, if $p>0$. 
\item[\em v)] Assume that there exists a line bundle $L$ on $Y$ and an integer $s\geq 2$ such that 
$\mathscr O_Y(1)\cong L^{\otimes s}$. If $p>0$ assume moreover that $s$ is prime to $p$. Then $Y$ is not a set-theoretic complete intersection in $\mathbb P^n$.
\item[\em vi)] Assume that $Y$ is a set-theoretic complete intersection of dimension $\geq 3$. If $p=0$ then the restriction map $\Pic(\mathbb P^n)\to\Pic(Y)$ is an isomorphism. If $p>0$ and $Y$ is nonsingular, then $\Pic(Y)/\mathbb Z[\mathscr O_Y(1)]$ is a finite $p$-group $($and in particular, $\rank\Pic(Y)=1)$.
\end{enumerate}}

\medskip

In some special cases, parts of Theorem 1 are known. To our best knowledge the approach to prove Theorem 1, which is based on Grothendieck-Lefschetz theory, is new. For instance, if $Y$ is a nonsingular closed subvariety of the complex projective space $\mathbb P^n_{\mathbb C}$, part ii) is an old result of Hartshorne \cite{Ha3}, while part iii) is new. Part vi) is known in characteristic zero (see \cite{N2}). However, we give another proof based on a result of Grothendieck \cite{[SGA2]} (see also Theorem \ref{gro2}  below). In Section 3 we prove Theorem 1 using Grothendieck-Lefschetz theory \cite{[SGA2]}, although parts ii) and iii) also require some basic results from the theory of Picard schemes, see Grothendieck \cite{[FGA]}. 

In Section 3, we apply Theorem 1 to provide several examples (in any characteristic) of subvarieties $Y\subseteq \mathbb P^n$ of dimension $\geq 2$ that cannot be set-theoretic complete intersections in $\mathbb P^n$.

In Section 4 we determine the arithmetic rank of the rational normal scolls. Specifically, given the integers $\{d, n_1,\ldots n_d\}$ such that $d\geq 2$ and $n_i\geq 1$, $i=1,\ldots,d$, let us consider
the $d$-dimensional rational normal scroll 
$$S_{n_1,\ldots,n_d}:=\mathbb P(\mathscr O_{\mathbb P^1}(n_1)\oplus\cdots\oplus\mathscr O_{\mathbb P^1}(n_d))$$
embedded in $\mathbb P^N$ via the very ample complete linear system $|\mathscr O_{\mathbb P(\mathscr O_{\mathbb P^1}(n_1)\oplus\cdots\oplus\mathscr O_{\mathbb P^1}(n_d))}(1)|$, where  $N:=\sum_{i=1}^dn_i+d-1$ and $\mathbb P(\mathscr O_{\mathbb P^1}(n_1)\oplus\cdots\oplus\mathscr O_{\mathbb P^1}(n_d))$
is the projective bundle associated to the vector bundle $\mathscr O_{\mathbb P^1}(n_1)\oplus\cdots\oplus\mathscr O_{\mathbb P^1}(n_d)$ over $\mathbb P^1$.

We prove the following result (see  Theorem \ref{mainscroll} and  Corollary \ref{mainscroll1} below):

\medskip
\noindent{\bf Theorem 2} {\em Under the above notation and assumptions, the arithmetic rank of $S_{n_1,\ldots,n_d}$ in $\mathbb P^N$  is $N-2=\sum_{i=1}^dn_i+d-3$. In particular, $S_{n_1,\ldots,n_d}$ is a set-theoretic complete intersection in $\mathbb P^N$
if and only if $\dim(S_{n_1,\ldots,n_d})=2$.}

\medskip
The fact that the $2$-dimensional rational normal scrolls $S_{n_1,n_2}$ are set-theoretic complete intersections in  $\mathbb P^{n_1+n_2+1}$ was already known, see   Valla \cite{Va} and  Robbiano-Valla \cite{RoVa} in some special cases, and subsequently, Verdi \cite{V1} in general.  Our approach provides in particular a new proof of the result of Verdi 
\cite{V1} for the two-dimensional rational normal scrolls. Moreover, in general our homogeneous equations are of lower degree than Verdi's equations. 

As far as the proof of Theorem 2 is concerned, we notice that the inequality $\ara(S_{n_1,\ldots,n_d})\geq N-2$ is of topological nature. In fact, in characteristic zero this inequality is a consequence of a generalization (due to Lazarsfeld \cite{La}) of a topological result of Sommese \cite{So} (see Corollary \ref{laz1} below), while in positive characteristics it follows from an analogous result in the \'etale cohomology, essentially due to Lyubeznik \cite{Ly} (see Theorem \ref{etale} below). 

So the problem is reduced to proving the reverse inequality $\ara(S_{n_1,\ldots,n_d})\leq N-2$.  And this is done  by exhibiting  $N-2$ explicit homogeneous equations defining $S_{n_1,\ldots,n_d}$ set-theoretically in $\mathbb P^N$.

 Throughout this paper we shall fix an algebraically closed field $k$ of characteristic $p\geq 0$. All algebraic varieties that will occur will be defined over $k$. The terminology and the notation used are standard, unless otherwise explicitly stated.  
 \medskip
 
 \noindent{\bf Acknowledgement.}  The authors want to thank Aldo Conca for some useful discussions.

\section{Background material}\label{first}\addtocounter{subsection}{1}

In this section we recall some well-known theorems that will be used in the sequel. We start with some basic facts from  Grothendieck-Lefschetz theory (see \cite{[SGA2]}).

\begin{definition}[Grothendieck \cite{[SGA2]}]\label{(2.2.1)} {\em Let $Y$ be a closed subvariety 
of a projective variety $X$. We say that
the pair $(X,Y)$ satisfies the {\it Grothendieck-Lefschetz condition}
$\Lef(X,Y)$ if for every open subset $U$ of $X$ containing $Y$ and for every vector
bundle $E$ on $U$ the natural map $H^0(U,E)\to H^0(X_{/Y},\h E)$ is an
isomorphism, where $X_{/Y}$ is the formal completion of $X$ along $Y$, $\pi\colon X_{/Y}\to U$ the canonical morphism, and  $\h E:=\pi^*(E)$. We also say that $(X,Y)$ satisfies the {\it effective
Grothendieck-Lefschetz condition} $\Leff(X,Y)$ if the Grothendieck-Lefschetz
condition $\Lef(X,Y)$ holds and, moreover, for every formal vector bundle
$\mathscr E$ on $X_{/Y}$ there exists an open subset $U$ of $X$ and a vector
bundle $E$ on $U$ such that $\mathscr E\cong\h E$.}\end{definition}

\begin{theorem}[Grothendieck \cite{[SGA2]}, or also \cite{Ha2}, Theorem 1.5, page 172]\label{gro} Let $X$ be closed irreducible subvariety of $\mathbb P^n$. Let $Y$ be a complete intersection subscheme of
$X$ with $r$ hyperplanes of $\mathbb P^n$, and assume that $\dim(Y)=\dim(X)-r\geq 2$. If in addition $Y$ is contained in the nonsingular locus of $X$, then
the effective Grothendieck-Lefschetz condition $\Leff(X,Y)$ holds.
\end{theorem}

\begin{theorem}[Grothendieck \cite{[SGA2]}, Expos\'e X, Th\'eor\`eme 3.10]\label{gro1} Let $Y$ be a closed subvariety of $\mathbb P^n$ such that the effective Grothendieck-Lefschetz condition $\Leff(\mathbb P^n, Y)$ holds. Then $Y$ is algebraically simply connected, i.e. there are no non-trivial connected \'etale covers of $Y$.\end{theorem}

\begin{theorem}[Grothendieck \cite{[SGA2]}, Expos\'e XII, Corollary 3.7]\label{gro2} Let $Y$ be a $($non necessarily reduced or irreducible$)$ subscheme of 
$\mathbb P^n$ of dimension $\geq 3$ which is a scheme-theoretic complete intersection in $\mathbb P^n$. Then the natural restriction map $\Pic(\mathbb P^n)\to \Pic(Y)$ is an isomorphism.\end{theorem}

In the last section we shall also make use of two further  Lefschetz type results. The first one regards the singular cohomology (see \cite{La}, (1.8)) and generalizes earlier results due to Sommese \cite{So} and Newstead \cite{N1} and \cite{N2} . Instead the second one uses the \'etale cohomology.

\begin{theorem}\label{laz} Let $X$ be a nonsingular projective variety over $\mathbb C$ of dimension $n\geq 2$, and let $E$ be an ample vector bundle of rank $e$ on $X$. Let $s\in\Gamma(X,E)$ be a global section of $S$ and let $Y=Z(s)$ be the zero locus of $s$. Then the natural restriction map of singular cohomology groups
$$H^i(X,\mathbb Z)\to H^i(Y,\mathbb Z)$$
is an isomorphism for every $i<n-e$, and injective if $i=n-e$.
\end{theorem}

Using Theorem \ref{laz}, the exponential sequences for $X$ and for $Y$ and Serre's GAGA one immediately gets:

\begin{corollary}\label{laz1} Under the hypotheses of Theorem $\ref{laz}$, assume that $n-e\geq 3$. Then the natural restriction map $\Pic(X)\to\Pic(Y)$ is an isomorphism.\end{corollary}

\begin{rem*} Under the extra-hypothesis that $\dim(Y)=n-e$, Theorem \ref{laz} was proved by Sommese
in \cite{So}. Actually, Lazarsfeld observed in \cite{La} that essentially the same proof of Sommese also yields the general case when $\dim(Y)\geq n-e$ (and in the last section we are going to use this result exactly under this more general assumption). On the other hand, Newstead proved various Lefschetz type results (for homotopy groups, and singular homology and cohomology groups) in the case where $E$ is a direct sum of line bundles of the form $\mathscr O_X(m)$, with $m>0$, see \cite{N1} and \cite{N2}.\end{rem*}

The next  theorem (which follows easily from some results of Lyubeznik \cite{Ly}) takes care of the case when the characteristic of $k$ is arbitrary. 

\begin{theorem}\label{etale} Assume that $p>0$ and let $Y$ be a nonsingular closed subvariety of $\mathbb P^N$ which is set theoretically given by $s$ equations, with $s\leq N-3$. Then the restriction map
$\alpha\colon\Pic(\mathbb P^N)\to\Pic(Y)$ is injective and $\Coker(\alpha)$ is a finite torsion $p$-group.\end{theorem}

\proof It is a basic fact that the \'etale cohomological dimension of an affine variety $U$ is $\leq\dim(U)$
(see e.g. Milne \cite{Mil}, Theorem 15.1; this result is in fact an \'etale analogue of a classical topological result of Andreotti and Frankel, see \cite{AF}, cf. also Milnor \cite{Miln}, page ?). If instead $U$ is covered by $s$ open affine subsets, then this result plus repeated application of Mayer-Vietoris (see \cite{Mil}, Theorem 10.8) yield the fact that the \'etale cohomological dimension of $U$ is $\leq\dim(U)+s-1$. Applying this to $U:=\mathbb P^N\setminus Y$ (which by hypothesis is covered by $s$ affines, namely the complements of the surfaces of the $s$ equations defining $Y$ set-theoretically)
we get that the \'etale cohomological dimension of $\mathbb P^N\setminus Y$ is $\leq N+s-1\leq 2N-4$.
At this point we can apply Lemma 11.1 in \cite{Ly} to get the conclusion.\qed

\begin{rem*} A Lefschetz type result similar to Theorem \ref{etale}, but for fundamental group has been proved by Cutkovsky in \cite{C}.\end{rem*}

\section{Necessary conditions for  set-theoretic complete \\ intersections}\label{second}\addtocounter{subsection}{1}

We start with the following result:  

\begin{prop}[cf. \cite{B},  p. 115] \label{tors} Let $Y$ be a closed subvariety of the projective irreducible variety $X$ over $k$, and assume that $p=0$. Then for every formal line bundle $\mathscr L\in\Pic(X_{/Y})$
such that $\mathscr L|Y\cong M^{\otimes s}$, with $s\geq 2$ an integer and
$M\in\Pic(Y)$, there exists a formal line bundle
$\mathscr M\in\Pic(X_{/Y})$ such that $\mathscr L\cong\mathscr M^{\otimes s}$
and $\mathscr M|Y\cong M$. The same statement holds if $p>0$, provided that $s$ is prime to $p$.
\end{prop}

\begin{proof} Since this result is going to be used later on in an essential way, for the convenience of the reader we include the proof. 
For every $n\geq 0$ consider the infinitesimal neighbourhood $Y(n)=(Y,\mathscr O_X/\mathscr I^{n+1})$
of order $n$ of $Y$ in $X$. We have the inclusions of subschemes 
$$Y(0)\subset Y(1)\subset Y(2)\subset\cdots\subset X.$$
Then giving a formal line bundle $\mathscr L$ on $X_{/Y}$ amounts to giving a 
sequence $\{L_n\}_{n\geq 0}$, with $L_n\in\Pic(Y(n))$ such that $L_{n+1}|Y(n)\cong
L_n$ for every $n\geq 0$. The hypothesis says that $L_0\cong M^{\otimes s}$ for some $M$ in
$\Pic(Y(0))=\Pic(Y)$. We shall construct by induction a formal
line bundle $\mathscr M=\{M_n\}_{n\geq 0}$ in $\Pic(X_{/Y})$ with the
desired properties. Starting with $M_0=M$, the induction step is the
following:

\medskip

\noindent{\it Claim.} Assume that for a fixed integer $n\geq 0$
there exists $M_n\in\Pic(Y(n))$ such that $L_n\cong M^{\otimes s}_n$.
Then there exists $M_{n+1}\in\Pic(Y(n+1))$ such that
$L_{n+1}\cong M^{\otimes s}_{n+1}$ and $M_{n+1}|Y(n)\cong M_n$.

\medskip

Indeed, consider the exact sequence of cohomology
$$H^1(Y,\mathscr I^{n+1}/\mathscr I^{n+2})\to\Pic(Y(n+1))\to\Pic(Y(n))\to
H^2(Y,\mathscr I^{n+1}/\mathscr I^{n+2})$$
associated to the truncated exponential exact sequence
$$0\to\mathscr I^{n+1}/\mathscr I^{n+2}\to\mathscr O_{Y(n+1)}^*\to\mathscr O_{Y(n)}^*\to 0\;,
$$
where $\mathscr I$ is the sheaf of ideals of $Y$ in $X$. To prove the
claim, observe that in this cohomology sequence the extreme terms are vector
spaces over $k$; in particular $H^2(Y,\mathscr I^{n+1}/\mathscr I^{n+2})$ has
no torsion because $\car(k)=0$. Then the class of $M_n$ in
$$\Pic(Y(n))/\Im(\Pic(Y(n+1))\to\Pic(Y(n)))
\subseteq H^2(Y,\mathscr I^{n+1}/\mathscr I^{n+2})$$
is a torsion element of order dividing $s$.
Since $H^2(Y,\mathscr I^{n+1}/\mathscr I^{n+2})$ has no torsion we infer that
$M_n\in\Im(\Pic(Y(n+1))\to\Pic(Y(n)))$, i.e. there exists $N\in\Pic(Y(n+1))$
such that $N|Y(n)\cong M_n$. Now
$$(L_{n+1}\otimes N^{\otimes(-s)})|Y(n)\cong L_n\otimes M_n^{\otimes(-s)}\cong\mathscr O_{Y(n)}.$$
Therefore $L_{n+1}\otimes N^{\otimes(-s)}$ is a line bundle on $Y(n+1)$ coming
from the $k$-vector space $H^1(Y,\mathscr I^{n+1}/\mathscr I^{n+2})$. Since 
$\car(k)=0$ every element of such a
$k$-vector space is divisible by $s$, whence 
$$L_{n+1}\otimes N^{\otimes(-s)}\cong P^{\otimes s},\;\;\text{with}\;\;P\in\Pic(Y(n+1))\;\;
\text{such that}\;\; P|Y(n)\cong\mathscr O_{Y(n)}.$$
If we take $M_{n+1}=N\otimes P$ we get $L_{n+1}\cong M_{n+1}^{\otimes s}$ and
$M_{n+1}|Y(n)\cong M_n$, which proves the claim.

The last assertion of the proposition comes from the above argument plus
the observation that the $k$-vector space $H^2(Y,\mathscr I^{n+1}/\mathscr I^{n+2})$
over a field $k$ of characteristic $p>0$ has no $e$-torsion and every element 
of the $k$-vector space $H^1(Y,\mathscr I^{n+1}/\mathscr I^{n+2})$ is (uniquely) divisible by
$e$, for every $e>0$ prime to $p$.\qed\end{proof}

\medskip

Here are two corollaries of Proposition \ref{tors}:

\begin{corollary}\label{tors'} Under the notation and hypotheses of Proposition $\ref{tors}$, the abelian group $\Coker(\Pic(X_{/Y})\to\Pic(Y))$ is torsion-free if $p=0$, and has no $e$-torsion for every positive integer $e$ which is prime to 
$p$, if $p>0$.\end{corollary}

If $A$ is an abelian (multiplicative) group with neutral element $e$,  we shall denote by $\Tors(A)$ the torsion subgroup of $A$. Let $p\geq 2$ be a prime integer. Then we also set
$$\Tors^p(A):=\{a\in A\ |\ \exists s>0\;\text{such that $s$ is prime to $p$ and}\;a^s=e\}.$$
Clearly $\Tors^p(A)$ is a subgroup of $A$. Then we also have:

\begin{corollary}\label{tors"} Under the notation and hypotheses of Proposition $\ref{tors}$ $($with $p=\car(k))$, assume furthermore that $X$ is nonsingular and $\Leff(X,Y)$ holds. Then:
\begin{enumerate} \item[\em i)] The abelian group $\Coker(\Pic(X)\to\Pic(Y))$ is torsion-free if $p=0$, and has no $s$--torsion for every positive integer $s$ which is prime to $p$ if $p>0$. If in addition  the restriction map $\alpha\colon\Pic(X)\to\Pic(Y)$ is injective $($this is always the case if $X=\mathbb P^n$ and $\dim(Y)>0)$, then $\alpha$ induces an isomorphism $\Tors(\Pic(X))\cong\Tors(\Pic(Y))$ if $p=0$, and an isomorphism $\Tors^p(\Pic(X))\cong\Tors^p(\Pic(Y))$ if $p>0$.
\item[\em ii)] Assume in addition that $Y$ meets every hypersurface of $X$. Let $L$ be a line bundle on $X$ such that  $L|Y\cong M^{\otimes s}$ for some $M\in\Pic(Y)$ and $s\geq 2$ prime to $p$, if $p>0$. Then there exists a line bundle $M'\in\Pic(X)$ such that $M'|Y\cong M$ and $L\cong {M'}^{\otimes s}$.\end{enumerate}
\end{corollary}

\proof  i) The canonical restriction map $\Pic(X)\to\Pic(Y)$ factors as $\Pic(X)\to\Pic(X_{/Y})\to\Pic(Y)$. By Corollary \ref{tors'} it is enough to show that the map $\Pic(X)\to\Pic(X_{/Y})$ is surjective. To check this, let $\mathscr L\in\Pic(X_{/Y})$ be an arbitrary formal line bundle. By $\Leff(X,Y)$, there exists an open subset $U$ of $X$ containing $Y$ and a line bundle $L'$ on $U$ such that $\hat{L'}\cong\mathscr L$. Since $X$ is nonsingular, $L'$ extends to a line bundle $L\in\Pic(X)$. Then clearly $\hat L=\hat{L'}\cong\mathscr L$, which yields i).

To prove ii) observe that the hypotheses that $X$ is nonsingular and $Y$ meets every hypersurface of $X$ implies that the restriction map $\Pic(X)\to\Pic(U)$ is an isomorphism for every open subset $U$ containing $Y$. Then ii) follows immediately from Proposition \ref{tors}. Indeed, by Proposition \ref{tors} and the fact that $\Leff(X,Y)$ holds we infer that there exists a formal line bundle $\mathscr M\in\Pic(X_{/Y})$ such that $\mathscr  M|Y\cong M$ and the formal completion $\hat L$ is isomorphic to  $\mathscr M^{\otimes s}$. By $\Leff(X,Y)$ again we find an open neighbourhood $U$ of $Y$ in $X$  and a line bundle $M^{\prime\prime}\in\Pic(U)$ such that $L|U\cong {M^{\prime\prime}}^{\otimes s}$ and
the formal completion $\hat{M^{\prime\prime}}$ is isomorphic to $\mathscr M$ (in particular, $M^{\prime\prime}|Y\cong M$ and $L|U\cong {M^{\prime\prime}}^{\otimes s}$). Finally, since the restriction map $\Pic(X)\to\Pic(U)$ is an isomorphism, we can (uniquely) extend $M^{\prime\prime}$ to a line bundle $M'\in\Pic(X)$ with the desired properties.\qed 

\begin{rem}\label{00} {\em The hypothesis that $X$ is nonsingular is essential in Corollary \ref{tors"}. Indeed, fix $r\geq 2$ and $s\geq 2$, and take $X\subset\mathbb P^{\binom{r+s}{s}}$ the projective cone over the polarized variety $(\mathbb P^r,\mathscr O_{\mathbb P^r}(s))$, e.g. the projective cone over the Veronese embedding $\mathbb P^r\hookrightarrow\mathbb P^{\binom{r+s}{s}-1}$. Take $Y$ the intersection of $X$ with the hyperplane at infinity. Then $Y\cong\mathbb P^r$, and since $Y$ is a hyperplane section of $X$ of dimension $r\geq 2$, by Grothendieck's result (Theorem \ref{gro} above) the effective Grothendieck-Lefschetz condition $\Leff(X,Y)$ holds. In this case $\Coker(\Pic(X)\to\Pic(Y))\cong \mathbb Z/s\mathbb Z$, and in particular, $\Coker(\Pic(X)\to\Pic(Y))$ has torsion if $p=0$. However, if $U:=X\setminus\{p\}$, with $p$ the vertex of the cone $X$, then $U$ is nonsingular and  the restriction map $\Pic(U)\to\Pic(Y)$ is an isomorphism. }\end{rem}

\begin{rems}\label{01} {\em \begin{enumerate}
\item[i)] The earliest reference we are aware of, regarding the torsion-freeness of the cokernel of
some natural restriction maps between singular cohomogy groups, is \cite{AF}. Specifically, let
$X$ be an $n$-dimensional  nonsingular subvariety of the complex projective space $\mathbb P^N(\mathbb C)$, and let $Y$ the proper intersection of $X$ with an hyperplane $H$ of $\mathbb P^N(\mathbb C)$. Then part of the famous topological theorem on hyperplane sections asserts that if $n\geq 2$ then the canonical restriction map $H^{n-1}(X,\mathbb Z)\to H^{n-1}(Y,\mathbb Z)$ is injective and its cokernel is torsion-free.
\item[ii)] Let $Y$ be a nonsingular (scheme-theoretic) complete intersection surface of $\mathbb P^N$ over a field $k$ of characteristic zero. Then  Robbiano proved in \cite{Ro} a criterion for a curve $C$ lying on $Y$ to be the scheme-theoretic intersection of $Y$ with a hyperplane $H$ of $\mathbb P^n$. In order to do that he used in an essential way the fact that the cokernel of the canonical map $\Pic(\mathbb P^N)\to\Pic(Y)$ is torsion-free. \end{enumerate}}\end{rems}

\begin{lemma}\label{leffc} Let $Y$ be a closed irreducible subvariety of $\mathbb P^n$ of dimension
$d\geq 2$. If $Y$ is a set-theoretic complete intersection in $\mathbb P^n$ then the effective Grothendieck-Lefschetz condition $\Leff(\mathbb P^n,Y)$ holds.\end{lemma}

\proof Let  $f_1,...,f_r\in k[T_0,T_1,...,T_{n}]$ be homogeneous polynomials defining $Y$ in $\mathbb P^{n}$ as a set-theoretic complete intersection in $\mathbb P^n$, where $r=n-d$. Then we have $\;\root\of{(f_1,\ldots,f_r)}=\mathscr I_+(Y)$. If $Y'$ is
the subscheme of $\mathbb P^{n}$ defined by the ideal 
 $(f_1,...,f_{r})$ then $Y'$ is a scheme-theoretic complete
intersection of $\mathbb P^{n}$ of dimension $\geq 2$, and hence
by Theorem \ref{gro}, $\Leff(\mathbb P^{n},Y')$ holds. Since
$Y'_{\red}=Y$, we infer that $\mathbb P^n_{/Y}=\mathbb P^n_{/Y'}$, and in particular,
$\Leff(\mathbb P^{n},Y)$ also holds. \qed

\medskip

Now we are ready to  prove the main result of this section.

\begin{theorem}\label{main} Let $Y$ be a closed irreducible
subvariety of $\mathbb P^n$ of dimension $\geq 2$ over an algebraically closed field $k$
of characteristic $p\geq 0$. 
\begin{enumerate}\item[\em i)] If $Y$ is a set-theoretic complete intersection in $\mathbb P^n$ then $Y$
is algebraically simply connected, i.e. there are no non-trivial connected \'etale covers of $Y$.
\item[\em ii)]
Assume  that $p=0$ and  $Y$ normal. If  $H^1(\mathscr O_Y)\neq 0$, then $Y$ is not a set-theoretic complete intersection in $\mathbb P^n$.
\item[\em iii)] Assume  that $p>0$ and  $Y$ is normal. If $H^1(\mathscr O_Y)\neq 0$ and the Picard scheme $\underline{\Pic}_Y^0$ of $Y$ is reduced, then $Y$ is not a set-theoretic complete intersection in $\mathbb P^n$. $($If for example $H^2(\mathscr O_Y)=0$, then  $\underline{\Pic}_Y^0$ is always reduced, see $\cite{[FGA]}$, \'Expos\'e $236$, Proposition
$2.10$, $ii))$.
\item[\em iv)]  Assume that $Y$ is a set-theoretic complete intersection in $\mathbb P^n$. Then the restriction map $\alpha\colon\Pic(\mathbb P^n)\to\Pic(Y)$ is injective and $\Coker(\alpha)$ is torsion-free if $p=0$, and has no $s$-torsion for every integer $s>0$ which is prime to $p$, if $p>0$. 
\item[\em v)] Assume that there exists a line bundle $L$ on $Y$ and an integer $s\geq 2$ such that 
$\mathscr O_Y(1)\cong L^{\otimes s}$. If $p>0$ assume moreover that $s$ is prime to $p$. Then $Y$ is not a set-theoretic complete intersection in $\mathbb P^n$.
\item[\em vi)] Assume that $Y$ is a set-theoretic complete intersection of dimension $\geq 3$. If $p=0$ then the restriction map $\Pic(\mathbb P^n)\to\Pic(Y)$ is an isomorphism. If $p>0$ and $Y$ is nonsingular, then $\Pic(Y)/\mathbb Z[\mathscr O_Y(1)]$ is a finite $p$-group $($and in particular, $\rank\Pic(Y)=1)$.
\end{enumerate}
\end{theorem}

\proof Part i) is a consequence of Lemma \ref{leffc} and of  
 \cite{[SGA2]}, Expos\'e X, Th\'eor\`eme 3.10.

\medskip

ii) Assume that $Y$ is a set-theoretic complete intersection in $\mathbb P^n$. Then by Lemma \ref{leffc}, $\Leff(\mathbb P^{n},Y)$ holds. Clearly, $Y$ meets every hypersurface of $\mathbb P^n$.  Then by Corollary \ref{tors"}, i), we get $\Tors(\Pic(\mathbb P^n))\cong\Tors(\Pic(Y))$. Since $\Pic(\mathbb P^n)=\mathbb Z$, we get $\Tors(\Pic(Y))=0$. But, under our hypotheses, this is absurd because $\Pic(Y)$ contains the subgroup $\Pic^0(Y)$ of isomorphism classes of line bundles on $Y$ which are algebraically trivial. Then $\Pic^0(Y)$ is the underlying set of the Picard scheme $\underline{\Pic}^0_Y$. 
The fact that $Y$ is normal implies that the Picard scheme $\underline{\Pic}^0_Y$ is proper over $k$, and in particular, $(\underline{\Pic}^0_Y)_{\red}$ is an abelian variety (see \cite{[FGA]}, \'Expos\'e 236, Th\'eor\`eme $2.1$, ii)).  In fact, since the characteristic of $k$ is zero, by a theorem of Chevalley, the abelian scheme $\underline{\Pic}^0_Y$ is reduced (see \cite{M}, Lecture 25, Theorem 1) and the tangent space $T_{\underline{\Pic}^0_Y,0}$ is isomorphic with $H^1(\mathscr O_Y)$ (see \cite{M}, Lecture 24), which by hypothesis is of dimension $q:=h^0(\mathscr O_Y)>0$. Then the abelian group $\Pic^0(Y)$ contains a lot of torsion (for example, if $e\geq 2$ is an integer, the $e$-torsion subgroup of $\Pic^0(Y)$  is isomorphic with $(\mathbb Z/e\mathbb Z)^{2q}\neq 0$ (see \cite{M1}, Chap. II,  \S7), which yields the desired contradiction because $\Tors(\Pic(\mathbb P^n))=0$.

Notice that ii) is also an easy consequence of i). In fact, take a non-trivial torsion element in $\Pic^0(Y)$ of order $m\geq 2$ which is prime to $p$ if $p>0$, i.e. a non-trivial line bundle $L\in\Pic(Y)$ such that $L^{\otimes m}\cong\mathscr O_Y$ for some
$m\geq 2$ (with $m$ is the least natural number with this property). Then $L$ produces the non-trivial connected cyclic \'etale cover $\tilde Y=\Spec(\oplus_{i=0}^{m-1}L^{\otimes i})$, and, in particular, $Y$ is not algebraically simply connected.

\medskip

iii) The proof in this case is almost identical with the proof of ii). The only difference is that in characteristic $p>0$ the Picard scheme may not be reduced. But this possibility is ruled out by our hypothesis. Moreover, the $e$-torsion subgroup of a $q$-dimensional abelian variety is still isomorphic with $(\mathbb Z/e\mathbb Z)^{2q}$, provided that $p$ does not divide $e$ (see \cite{M1}, Chap. II,  \S7).

\medskip

iv) Clearly, $Y$ meets every hypersurface of $\mathbb P^n$, and in particular the map $\alpha$ is injective (Corollary \ref{tors"}, i)). Moreover by the proof of i), since $Y$ is a set-theoretic complete intersection in $\mathbb P^n$ of dimension $\geq 2$, $\Leff(\mathbb P^n,Y)$ holds. Then the conclusion follows from Corollary \ref{tors"}, i). 

\medskip

v)   We have $\Coker(\Pic(\mathbb P^n)\to\Pic(Y))=\Pic(Y)/\mathbb Z[\mathscr O_Y(1)]$. Then the conclusion follows from iv). 

\medskip

vi) The result follows (in arbitrary characteristic) from Theorem \ref{etale}, while in characteristic zero -- from an old result of Sommese (see \cite{So}, Proposion (1.16)). However, if $p=0$ we shall give another proof using Theorem \ref{gro2} of Grothendieck. 

By Lefschetz's principle we may assume that $k=\mathbb C$.  Let $$f_1,...,f_r\in \mathbb C[T_0,T_1,...,T_{n}]$$ be homogeneous equations defining 
$Y$ set-theoretically in $\mathbb P^n$, and set $Y':=\Proj(\mathbb C[T_0,T_1,...,T_{n}])$. Since $\root\of{(f_1,\ldots,f_r)}=\mathscr I_+(Y)$, we have $Y'_{\red}=Y$, and in particular the underlying topological spaces of $Y'$ and $Y$ are the same. If $F$ is an algebraic (e.g. a coherent) sheaf on an algebraic scheme 
$Z$ over $\mathbb C$, we shall denote by $F^{\an}$ the analytic sheaf associated to $F$ in the sense of Serre's GAGA (see \cite{gaga}). Then we have the following commutative diagram
\begin{equation*}
\begin{CD}
{0} @>>> \mathbb Z_{Y'} @>>> \mathscr O_{Y'}^{\an} @>>> {\mathscr O_{Y'}^{\an}}^* @>>> {0}\\
@.  @V\id VV @VVV @VVV  \\
{0} @>>> \mathbb Z_{Y} @>>> \mathscr O_{Y}^{\an} @>>> {\mathscr O_{Y}^{\an}}^* @>>> {0}\\
\end{CD}\end{equation*}
where the rows are the exponential exact sequences of $Y'$ and $Y$ ($\mathbb Z_{Y'}$ and $\mathbb Z_Y$ are the constant sheaves on $Y'$ and on $Y$ respectively with stalks $\mathbb Z$) and the vertical arrows are the canonical restriction maps induced by the inclusion $Y\subseteq Y'$. The above diagram yields the following commutative diagram with exact rows
\begin{equation*}
\begin{CD}
H^1(\mathscr O_{Y'}^{\an}) @>>> H^1({\mathscr O_{Y'}^{\an}}^*) @>>> H^2(Y',\mathbb Z) @>>> H^2(\mathscr O_{Y'}^{\an})\\
 @VVV @VVV @VV\id V  @VVV\\
H^1(\mathscr O_{Y}^{\an}) @>>> H^1({\mathscr O_{Y}^{\an}}^*) @>>> H^2(Y,\mathbb Z) @>>> H^2(\mathscr O_{Y}^{\an})\\
\end{CD}\end{equation*}
Since $Y'$ is a scheme-theoretic complete intersection in $\mathbb P^n$ of dimension $\geq 3$ we get $H^i(\mathscr O_{Y'})=0$ for $i=1,2$. Then by Serre's GAGA
\cite{gaga} we also get $H^i(\mathscr O_{Y'}^{\an})=0$ for $i=1,2$. Moreover, Serre's GAGA, also implies
$$H^1({\mathscr O_{Y'}^{\an}}^*)\cong H^1(\mathscr O_{Y'}^*)=\Pic(Y').$$
Doing the same thing for the cohomology groups on the bottom row of the last diagram and observing that since $Y$ is a set-theoretic complete intersection of dimension $\geq 2$, by ii) we have $H^1(\mathscr O_Y)=0$, by putting things together we get the commutative diagram
\begin{equation*}
\begin{CD}
0 @>>> \Pic(Y') @>>> H^2(Y',\mathbb Z) @>>> 0\\
 @. @VVV @VV\id V  \\
0 @>>> \Pic(Y) @>>> H^2(Y,\mathbb Z) \\
\end{CD}\end{equation*}
From this diagram it follows that the canonical restriction map $\Pic(Y')\to\Pic(Y)$ is an isomorphism. Finally,  since $Y'$ is a scheme-theoretic complete intersection of dimension $\geq 3$, by Theorem \ref{gro2} of Grothendieck, the restriction map $\Pic(\mathbb P^n)\to \Pic(Y')$ is an isomorphism. Therefore the restriction map $\Pic(\mathbb P^n)\to \Pic(Y)$ is an isomorphism, i.e. $\Pic(Y)=\mathbb Z[\mathscr O_Y(1)]$.\qed

\begin{rems}\label{verdi} 
 {\em \begin{enumerate}
\item[i)] If $k=\mathbb C$ and $Y$ is nonsingular, part ii) of Theorem \ref{main} is an old result of Hartshorne  (see \cite{Ha3}, Corollary 8.6). Our proof of this more general result contained in ii) and in iii)  is completely different from Hartshorne's proof (loc. cit.).
\item[ii)] Newstead proved in \cite{N1} and in  \cite{N2} topological Lefschetz theorems (for singular cohomology with coefficients in $\mathbb Z$)  for submanifolds $Y$ of the complex projective space $\mathbb P^n$, which are  defined by ``not too many equations'' in $\mathbb P^n$. As applications he gave several examples of submanifolds $Y$ of the complex projective space $\mathbb P^n$ which are not set-theoretic complete intersections.
\item[iii)] In Theorem \ref{main}, vi), the finite $p$-group $\Pic(Y)/\mathbb Z[\mathscr O_Y(1)]$ may effectively be non-trivial (see Proposition \ref{p=2} in the next section).
\end{enumerate}}
\end{rems}

\section{Examples of projective varieties that are not set-theoretic  complete intersections}\label{third}\addtocounter{subsection}{1}

{\bf The Veronese embedding.} $\;$The image of the $s$-fold Veronese embedding of $\mathbb P^1$
in $\mathbb P^s$ (the rational normal curve of degree $s\geq 2$ in $\mathbb P^s$) is known to be a set-theoretic complete intersection in $\mathbb P^s$ (this is completely elementary, see \cite{V}, cf. also \cite{RoVa}), but not a scheme-theoretic complete intersection if $s\geq 3$ (because it is not subcanonical). 

On the other hand, for every integers $r,s\geq 2$ consider the $s$-fold Veronese
embedding $i\colon\mathbb P^r\hookrightarrow\mathbb P^{n(r,s)}$ over an algebraically
closed field of characteristic $p\geq 0$, with $n(r,s)=
\binom{r+s}{s}-1$. Then $Y:=i(\mathbb P^r)$ is not a set-theoretic
complete intersection in $\mathbb P^{n(r,s)}$, provided that $p$ does not divide $s$, if $p>0$. This follows immediately from Theorem \ref{main}, iv), because $\mathscr O_{\mathbb P^{n(r,s)}}(1)|Y=\mathscr O_{\mathbb P^r}(s)$ and $s\geq 2$ and $s$ is prime to $p$, if $p>0$. (These facts have already been noticed in \cite{B}, page 116.)

The situation when $p>0$ and $p|s$ is rather interesting (in the sense that in some cases $Y$ may be a set-theoretic complete intersection). Precisely, one has the following result:

\begin{prop}[Gattazzo \cite{Ga}]\label{p=2} Assume that $p>0$ and $r\geq 2$, and let $s=p^m$ be a positive power of $p$. Then the image $Y$ of the Veronese embedding 
$i\colon\mathbb P^r\hookrightarrow\mathbb P^{n(r,s)}$ is a set-theoretic $($but not a scheme-theoretic$)$ complete intersection.\end{prop}

\begin{rems}\label{ga} {\em \begin{enumerate} \item[i)]  Assume that $p>0$ and $r\geq 2$. Then by Proposition \ref{p=2} the image $Y$ of the $p^m$-fold Veronese embedding $\mathbb P^r\hookrightarrow\mathbb P^{n(r,p^m)}$ is a set-theoretic complete intersection in $\mathbb P^{n(r,p^m)}$,
with $n(r,p^m)=\binom{r+p^m}{p^m}$.  On the other hand, in the hypotheses of Proposition \ref{p=2}, $\Coker(\alpha)=\mathbb Z/p^m\mathbb Z$. This shows in particular that in Theorem \ref{main}, iv), $\Coker(\alpha)$ may be a non-trivial finite $p$-group if 
$p>0$ (compare with Corollary \ref{tors"}, i)), and also that Theorem \ref{main}, vi) is false in general in positive characteristic. 
\item[ii)] From Proposition \ref{p=2} and the above arguments we infer that $\Leff(\mathbb P^{n(r,p^m)},Y)$ does hold, where $Y$ is the image of the $p^m$-fold Veronese embedding $\mathbb P^r\hookrightarrow\mathbb P^{n(r,p^m)}$ over an algebraically closed field $k$ of characteristic $p>0$. This is in contrast with the case $p=0$ when $\Leff(\mathbb P^{\frac{r(r+3)}{2}},Y)$ never holds.
\item[iii)] Gattazzo proved in \cite{Ga} an even more general result than Proposition \ref{p=2}. Namely,  he showed that also some projections of the $p^m$-fold Veronese embedding are set-theoretic complete intersections if $p>0$. For instance, the projection $Y\subset\mathbb P^4$  of the Veronese surface in $\mathbb P^5$ from a general point of $\mathbb P^5$ is also a set-theoretic complete intersection in $\mathbb P^4$ if the characteristic of $k$ is $2$.\end{enumerate}}
\end{rems}

\noindent{\bf The Segre embedding.} Let $i\colon\mathbb P^m\times\mathbb P^n\hookrightarrow\mathbb P^{mn+m+n}$ be the Segre embeddimg of $\mathbb P^m\times\mathbb P^n$, with $m,n\geq 1$ and $m+n\geq 3$. Assume that the ground field is $\mathbb C$.
Then $Y:=i(\mathbb P^m\times\mathbb P^n)$ is not a set-theoretic complete intersection in  $\mathbb P^{mn+m+n}$. Indeed this follows from Theorem \ref{main}, vi), because $\Pic(\mathbb P^{mn+m+n})\cong\mathbb Z$ and $\Pic(\mathbb P^m\times\mathbb P^n)\cong\mathbb Z\times\mathbb Z$ (see also \cite{La}). However, much more is known in this case. Namely, Bruns and Schw\"anzl proved in \cite{BS} (see also \cite{Br} and \cite{BV}) that the arithmetic rank of the variety defined by the $(t\times t)$-minors of a generic $(p\times q)$-matrix is $pq-t^2+1$. If we take $t=2$, $p=m+1$ and $q=n+1$, we find that the arithmetic rank of $Y:=i(\mathbb P^m\times\mathbb P^n)$ is $pq-t^2+1=mn+m+n-2$.
Notice also that in the case when $m$ is arbitrary and $n=1$ this result is also a consequence  of Theorem 2 of the Introduction (cf. also Corollary \ref{scrolln1} below).

\medskip
\medskip

\noindent{\bf Examples of surfaces that are not set-theoretic complete intersections.}

 \medskip
 
{\bf i) Surfaces with geometric genus zero.}  Let $Y$ be any ruled nonrational surface (not necessarily minimal) over an algebraically closed field of arbitrary characteristic, i.e. $Y$ is birationally equivalent to $B\times\mathbb P^1$, with $B$ a nonsingular projective curve $B$ of genus $g>0$. Consider an arbitrary projective embedding $Y\hookrightarrow\mathbb P^n$. We have $h^1(\mathscr O_Y)=g>0$ and 
$H^2(\mathscr O_Y)=0$. Therefore, by Theorem \ref{main}, i) and ii), $Y$ is not a set-theoretic complete intersection in $\mathbb P^n$.  In particular, we obtain the following fact (proved in \cite{SW} using ad hoc arguments:  De Rham cohomology if $p=0$ and the \'etale cohomology if $p>0$): if $E\subset\mathbb P^2$ is an elliptic curve, then $Y:=E\times\mathbb P^1\subset\mathbb P^2\times\mathbb P^1\subset\mathbb P^5$ (via the Segre embedding) is not a set-theoretic complete intersection in $\mathbb P^5$. 

In fact, the same as above holds  if, instead of taking a ruled nonrational surface, we take any nonsingular projective surface $X$ with geometric genus $p_g=h^2(\mathscr O_Y)=0$ and
irregularity $q=h^1(\mathscr O_Y)>0$. For example a hyperelliptic surface $Y$; this is a surface with invariants $p_g=0$, $q=1$, 
$b_1=b_2=2$, $\chi(\mathscr O_Y)=0$ and Kodaira dimension $\kappa(Y)=0$ (see e.g. \cite{B1}). Such a surface $Y$ has the property that the Picard scheme is always reduced and the Albanese map $f\colon Y\to\Alb(Y)=B$ has the following properties: $B$ is an elliptic curve, every fiber of $f$ is an elliptic curve, and there is a second elliptic fibration 
$Y\to\mathbb P^1$ (loc. cit.).

\medskip

{\bf ii) Enriques surfaces.} Let $Y$ be an Enriques surface embedded in $\mathbb P^n$ over $k$ and assume that
$p\neq 2$. Then $Y$ is not a set-theoretic complete intersection in $\mathbb P^n$. Indeed, in this case $\Pic(Y)$ contains a non-trivial element of order $2$, namely the canonical class $\mathscr O_Y(K)$ (and in particular, is not algebraically simply connected because $\mathscr O_X(K)$ produces the cyclic non-trivial \'etale cover of $Y$ of degree $2$).  Then the conclusion follows from Theorem \ref{main},  i). Alternatively, $[\mathscr O_Y(K)]\not\in\Im(\Pic(\mathbb P^n)\to\Pic(Y))$ (because $\Pic(\mathbb P^n)=\mathbb Z$), whence $[\mathscr O_Y(K)]$ defines a non-trivial element of order $2$ in $\Coker(\Pic(\mathbb P^n)\to\Pic(Y))$. Since the characteristic of $k$ is $\neq 2$, the conclusion also follows from Corollary \ref{tors"}, i).

\medskip

{\bf iii) Ruled nonrational surfaces with rational singularities.} Let $X$ be a nonsingular ruled nonrational surface over  an algebraically closed field $k$ of characteristic zero,  and assume that $p\geq 0$ is arbitrary. Let $\pi\colon X\to B$ be the canonical ruled fibration, with $B$ a nonsingular projective curve of genus $g=h^1(\mathscr O_X)>0$, and assume that there exists at least one degenerate fiber (i.e. reducible) fiber $\pi^{-1}(b)$. Fix $m\geq 1$ points $b_1,\ldots,b_m\in B$ such that the fiber $\pi^{-1}(b_i)$ is degenerate for every $i=1,\ldots,m$. As is well known (see e.g. \cite{B2}, Lemma 7), if for every $i=1,\ldots,m$
we are given a closed connected curve $\varnothing\neq  Z_i\subsetneq\pi^{-1}(b_i)$ is a closed connected curve of $\pi^{-1}(b_i)$, then there exists a birational morphism $f\colon X\to Y$, with $Y$ a normal projective surface such that:

$\bullet$ $f(Z_i)$ is a point of $y_i\in Y$, $i=1,\ldots,m$, and the restriction $$f':=f|Z_1\cup\ldots\cup Z_m\colon X\setminus (Z_1\cup\ldots\cup Z_m)\to Y\setminus\{y_1,\ldots,y_m\}$$ is a biregular isomorphism.

$\bullet$ The singularities $y_i\in Y$ are rational, $i=1,\ldots,m$, i.e. $R^1f_*(\mathscr O_X)=0$, and in particular,
$h^1(\mathscr O_Y)= h^1( \mathscr O_X)>0$ and $H^2(\mathscr O_Y)= H^2(\mathscr O_X)=0$. Moreover, the point $y_i$ is effectively singular on $Y$ if $Z_i$ is not an exceptional curve of the first kind on $X$.

$\bullet$ The morphism $\pi\colon X\to B$ factors uniquely as $\pi=\pi'\circ f$, con $\pi'\colon Y\to B$. 

\medskip

Now, let $Y\hookrightarrow\mathbb P^n$ be any projective embedding. Then by Theorem \ref{main}, ii) and iii) we deduce that $Y$ is not a set-theoretic intersection in $\mathbb P^n$. 

\medskip

{\bf iv)  Nonruled normal surfaces.}  Let $B$ be a smooth projective curve of genus $g\geq 1$ over $k$. Fix $2g$ distinct points $x,\;y_1,\ldots,y_{2g-1}\in B$. Let $f\colon X\to B\times B$ be the blowing up morphism of $B\times B$ of centers the $2g-1$ proints $(x,y_1),\;\ldots,(x,y_{2g-1})\in B\times B$, and let $C$ be the strict transform of the curve $\{x\}\times B$ via $f$. Clearly, $u:=f|C$ yields an isomorphism $C\cong B$.

We shall show that there exists a birational morphism $g\colon X\to Y$, with $Y$ a normal projective surface, which blows down the curve $C$ to a point of $Y$. 
In order to do that, we firstly observe that $C^2=1-2g<0$ (by the construction of the curve $C$).  

On the other hand, in the commutative diagram
$$\begin{CD}
H^1(\mathscr O_{B\times B})@>>>H^1(\mathscr O_{\{x\}\times B})\\
@Vf^*VV @ VVu^* V\\
H^1(\mathscr O_X)@ >>> H^1(\mathscr O_C)\\
\end{CD}$$
the vertical maps are isomorphisms (since $f$ is the blowing up morphism of $B\times B$ of center finitely many nonsingular points), and the top horizontal arrow is surjective because the inclusion $\{x\}\times B\hookrightarrow B\times B$ is a section of the second projection of $B\times B$. It follows that the bottom horizontal map is also surjective. Let $\mathscr O_X(-C)$ be the ideal sheaf of $C$ in $\mathscr O_X$. 

We claim that the restriction maps $H^1(\mathscr O_{(i+1)C})\to H^1(\mathscr O_{iC})$ are isomorphisms for every $i\geq 1$, where  $iC$  is the $i$-th infinitesimal neighbourhood of $C$ in $X$. This follows from the cohomology exact sequence
{\small $$H^1(\mathscr O_X(-iC)/\mathscr O_X(-(i+1)C))\to H^1(\mathscr O_{(i+1)C})\to H^1(\mathscr O_{iC})\to H^2(\mathscr O_X(-iC)/\mathscr O_X(-(i+1)C)),$$}if we show that the first and the last vector space are zero. The last vector space is clearly zero because $C$ is a curve. The first vector space is zero because $C^2=1-2g$ implies that $\deg(\mathscr O_X(-iC)/\mathscr O_X(-(i+1)C))=i(2g-1)\geq 2g-1$ if $i\geq 1$. Recalling that the bottom horizontal arrow in the above diagram is surjective,  by induction we infer that the restriction maps 
\begin{equation}\label{y}H^1(\mathscr O_X)\to H^1(\mathscr O_{iC})\end{equation} 
are surjective for every $i\geq 1$. 

Moreover, we claim that the Picard scheme $\underline{\Pic}^0_X$ is always reduced. In charcateristic zero this holds by a very general theorem of Cartier (see \cite{[FGA]},  or also \cite{M}, Lecture 25). If instead $p>0$, we have $h^1(\mathscr O_X)=h^1(\mathscr O_{B\times B})=2g$ (by 
K\"unneth), and since $H^1(\mathscr O_X)$ is canonically identified with the tangent space to $\underline{\Pic}^0_X$ at the origin, we get $\dim(\underline{\Pic}^0_X)\leq 2g$. Moreover, this inequality is strict if and only if  $\underline{\Pic}^0_X$ is nonreduced. On the other hand, as is well known, the dual abelian variety of the abelian variety $(\underline{\Pic}^0_X)_{\red}$ is the Albanese variety 
$\Alb(X)=\Alb(B\times B)$, which is isomorphic to $\Alb(B)\times\Alb(B)$. Hence $\dim(\underline{\Pic}^0_X)_{\red}=2g$ because $\Alb(B)$ is the Jacobian of $B$ and its dimension is $g$. Putting things together it follows that $\dim(\underline{\Pic}^0_X)=2g$ and $\underline{\Pic}^0_X$ is reduced.

Now, the surjectivity of \eqref{y}, the inequality $C^2<0$  and the fact that the Picard scheme $\underline{\Pic}^0_X$ is reduced allow us to apply Theorem 
14.23 of \cite{B1} to deduce that (in arbitrary characteristic) there exists a birational morphism $g\colon X\to Y$, with $Y$ a normal projective surface such that:

\medskip

$\bullet$ The image $g(C)$ is a point of $y\in Y$, and

$\bullet$ The restriction $X\setminus C\to Y\setminus\{y\}$ of $g$ is a biregular isomorphism.

\medskip

Notice that the projectivity of $Y$ in the conclusion of Theorem 14.23 of \cite{B1} is the main point. 
The surface $Y$ is going to be our example. We only need to show
 that $H^1(\mathscr O_Y)\neq 0$. To see this, consider the canonical  exact sequence in low degrees
\begin{equation}\label{yy}0\to H^1(Y,\mathscr O_Y)\to H^1(X,\mathscr O_X)\to H^0(Y,R^1f_*(\mathscr O_X))=R^1g_*(\mathscr O_X)_y\to 0\end{equation}
associated to the spectral sequence 
$$E^{p,q}_2=H^p(Y, R^qg_*(\mathscr O_X))\Longrightarrow H^{p+q}(X,\mathscr O_X).$$
By Grothendieck-Zariski's  theorem on formal functions together with  the fact (proved above) that $H^1(\mathscr O_{iC})\cong H^1(\mathscr O_C)$ for every $i\geq 1$, we have
$$R^1g_*(\mathscr O_X)_y=\inv\lim_{i\in\mathbb N}H^1(\mathscr O_{iC})\cong H^1(\mathscr O_C), $$
whence the exact sequence  \eqref{yy} becomes
$$0\to H^1(\mathscr O_Y)\to H^1(\mathscr O_X)\to H^1(\mathscr O_C)\to 0.$$
Since $h^1(\mathscr O_C)=h^1(\mathscr O_B)=g$ and  $ h^1(\mathscr O_X)=h^1(\mathscr O_{B\times B})=2g$, we get $h^1(\mathscr O_Y)=g>0$.

Finally, let $Y\hookrightarrow\mathbb P^n$ be an arbitrary projective embedding of $Y$. Then by Theorem \ref{main}, ii) and iii) we deduce that $Y$ is not a set-theoretic intersection in $\mathbb P^n$. Notice that in this example the surface $Y$ is birationally equivalent to an abelian surface if $g=1$, and to a surface of general type if $g\geq 2$. 

\medskip

{\bf v) Nonnormal surfaces.} Let $C$ be an irreducible curve over $k$, which is obtained from its normalization $\tilde C$ by identifying $n+1$ distinct points $P_0,\ldots,P_n$, with $n\geq 1$ (in the terminology of Serre \cite{gral}, chap. IV, $C$ is defined by the module $\sum_{i=0}^nP_i$; in the classical terminology, the singularity of $C$ is an ordinary $(n+1)$-fold point with $(n+1)$ distinct tangents). For instance, if $n=1$ then $C$ has just one singularity, which is an ordinary double point with distinct tangents. Then by Oort \cite{Oo}, Proposition (2.3), there is an exact sequence 
$$0\to\mathbb G_m^{\oplus n}\to\underline{ \Pic}^0_C\to\underline{\Pic}^0_{\tilde C}\to 0$$
of algebraic groups,
where $\mathbb G_m=k\setminus\{0\}$ is the multiplicative group of $k$. Since $\Tors(\mathbb G_m)\neq 0$, it follows that  $\Tors(\Pic(C))\neq 0$. 

Now, let $E$ be a vector bundle of rank $r\geq 2$ on $C$ and consider the projective bundle $Y:=\mathbb P(E)$ associated to $E$. Since $\Pic(Y)\cong\Pic(C)\oplus\mathbb Z$, it follows that $\Tors(\Pic(Y))=\Tors(\Pic(C))$. Let $Y\hookrightarrow\mathbb P^N$ be any projective embedding of $Y$. Then by Corollary \ref{tors"}, i), $Y$ cannot be a set-theoretic complete intersection in $\mathbb P^N$. This example has some interest because if we assume that the curve $\tilde C$ is rational, then $Y$ is a singular (nonnormal) rational projective variety of dimension $r\geq 2$.

\section{The arithmetic rank of rational normal scrolls}\label{fourth}\addtocounter{subsection}{1}

Let $E$ be an ample vector bundle of rank $d\geq 2$ over the projective line $\mathbb P^1$. By a well known theorem of Grothendieck, $E$ can be written as a direct sum of line bundles
$$E=\mathscr O_{\mathbb P^1}(n_1)\oplus\cdots\oplus\mathscr O_{\mathbb P^1}(n_d),$$ 
and since $E$ is ample,  $n_i>0$ for every $i=1,\ldots,d$. Let 
$$\mathbb P(E)=\mathbb P(\mathscr O_{\mathbb P^1}(n_1)\oplus\cdots\oplus\mathscr O_{\mathbb P^1}(n_d))$$ 
be the projective bundle associated to $E$. Since $n_i>0$ for every $i=1,\ldots,d$, the tautological line bundle $\mathscr O_{\mathbb P(E)}(1)$ is ample, and in fact, very ample. Consider the closed embedding $i\colon\mathbb P(E)\hookrightarrow\mathbb P^N$ associated to the very ample complete linear system $|\mathscr O_{\mathbb P(E)}(1)|$, with $N:=\sum_{i=1}^dn_i+d-1$. Then $S_{n_1,\ldots,n_d}:=i(\mathbb P(E))=i(\mathbb P(\mathscr O_{\mathbb P^1}(n_1)\oplus\cdots\oplus\mathscr O_{\mathbb P^1}(n_d)))$ is a  nonsingular $d$-dimensional subvariety  of $\mathbb P^N$, which is known to be arithmetically Cohen-Macaulay in $\mathbb P^N$; moreover,  $\Pic(S_{n_1,\ldots,n_d})\cong\Pic(\mathbb P(E)\cong\mathbb Z\oplus\mathbb Z$ (generated by the classes of $\mathscr O_{\mathbb P(E)}(1)$ and $\pi^*(\mathscr O_{\mathbb P^1}(1))$, where $\pi\colon \mathbb P(E)\to \mathbb P^1$  is the canonical projection of $\mathbb P(E)=\mathbb P(\mathscr O_{\mathbb P^1}(n_1)\oplus\cdots\oplus\mathscr O_{\mathbb P^1}(n_d))$). The subvariety $S_{n_1,\ldots,n_d}$ of $\mathbb P^N$ is called the $d$-dimensional {\em rational normal scroll}. 

The aim of this section is to prove the following result (see Theorem 2 of the Introduction):

\begin{theorem}\label{mainscroll} Under the above notation and assumptions, the arithmetic rank of $S_{n_1,\ldots,n_d}$ in $\mathbb P^N$  is $N-2=\sum_{i=1}^dn_i+d-3$. \end{theorem}

\begin{corollary}\label{mainscroll1} Under the notation of Theorem $\ref{mainscroll}$, $S_{n_1,\ldots,n_d}$ is a set-theoretic complete intersection in $\mathbb P^N$
if and only if  $S_{n_1,\ldots,n_d}$ is a surface $($i.e. $d=2)$. In particular, the two dimensional rational normal scroll $S_{n_1,n_2}$ is set-theoretic complete intersection in $\mathbb P^{n_1+n_2+1}$, but not  a scheme-theoretic complete intersection, unless $n_1=n_2=1$.\end{corollary}

\proof The first part is a direct consequence of Theorem \ref{mainscroll}. For the last part we notice that the canonical class of $S_{n_1,n_2}$ is given by
$$\omega_{\mathbb P(\mathscr O_{\mathbb P^1}(n_1)\oplus\mathscr O_{\mathbb P^1}(n_2))}=\pi^*(\mathscr O_{\mathbb P^1}(n_1+n_2-2))\otimes\mathscr O_{\mathbb P(\mathscr O_{\mathbb P^1}(n_1)\oplus\mathscr O_{\mathbb P^1}(n_2))}(-2),$$ 
whence $S_{n_1,n_2}$ is subcanonical in $\mathbb P^{n_1+n_2+1}$ if and only if $n_1=n_2=1$.\qed

\begin{rem*}The fact that the rational normal scrolls $S_{n_1,n_2}$ are set-theoretic complete intersections in  $\mathbb P^{n_1+n_2+1}$ was already known, see   Valla \cite{Va} and  Robbiano-Valla \cite{RoVa} in some special cases, and subsequently, Verdi \cite{V1} in general.  In particular, our approach also reproves (in a completely different way) the result of Verdi 
\cite{V1} for the two-dimensional rational normal scrolls. Moreover our method produces $n_1+n_2-1$ homogeneous equations defining $S_{n_1,n_2}$ as set-theoretic complete intersection in $\mathbb P^{n_1+n_2+1}$ which are in general of lower degrees with respect to the equations obtained in Verdi \cite{V1}. For example, if $n_1=n_2=2$, we prove that $S_{2,2}$ is the set-theoretic complete 
intersection of three hyperquadrics in $\mathbb P^5$, while Verdi needs two hyperquadrics and one hyperquartic.\end{rem*}

The proof of Theorem \ref{mainscroll} requires some preparation. 

We first recall that the rational normal curve $C_n$ of degree $n$ in $\mathbb{P}^n$, $n\ge 1,$ is defined  as the image of the Veronese map $\nu_n:\mathbb{P}^1\to \mathbb{P}^n$ sending $[\alpha,\beta]$ to $[\alpha^n,\alpha^{n-1}\beta,...,\alpha \beta^{n-1},\beta^n].$ It is  well known that $C_n$ may be realized as the locus of points which  give rank one to the matrix 
$$\begin{pmatrix}
X_0      & X_1 & . & . & .& X_{n-1}    \\
  X_1    &  X_2 & . & . & .& X_n
\end{pmatrix}.$$  Further, in \cite{RoVa} Valla and Robbiano, by using  Gr\"obner bases  theory,  showed that   $C_n$ is the set-theoretic complete intersection of the $n-1$ hypersurfaces defined by the following  polynomials 
\begin{equation}\label{set}  F_i= F_i(X_0,\ldots,X_n)=\sum_{\alpha=0}^i (-1)^{\alpha}{i\choose \alpha}X_{i+1}^{i-\alpha}X_{\alpha}X_i^{\alpha},\;\;i=1,\dots,n-1.\end{equation}
Notice  that it was Verdi who proved, see \cite{V}, for the first time  that $C_n$    is a set-theoretic complete intersection in $\mathbb P^n$. However, her equations and methods  are different  from those used by Robbiano and Valla, who found slightly simpler equations.

For every integers $d\ge 2$ and  $n_1,n_2,\dots,n_d>0$ as  above,  the $d$-dimensional rational normal scroll $S_{n_1,\dots,n_d}$ can also be described as the  rank one determinantal variety associated to the matrix
\begin{equation*}
A=\left(
\begin{array}{cccc}
X_{1,0}&X_{1,1}&\ldots&X_{1,n_1-1}\\
X_{1,1}&X_{1,2}&\ldots&X_{1,n_1}\\
\end{array}
\bigg|
\begin{array}{ccc}
\cdot&\cdot&\cdot\\
\cdot&\cdot&\cdot\\
\end{array}
\bigg|
\begin{array}{cccc}
X_{d,0}&X_{d,1}&\ldots&X_{d,n_d-1}\\
X_{d,1}&X_{d,2}&\ldots&X_{d,n_d}\\
\end{array}
\right)
\end{equation*}
i.e. a matrix consisting of $d$ blocks of sizes $2\times n_1,...,2\times n_d$ respectively,  with each block a generic catalecticant matrix 
(see \cite{H}, pp. 105--109). These blocks correspond to the canonical decomposition
$$H^0(\mathbb P(E),\mathscr O_{\mathbb P(E)}(1))=H^0(\mathbb P^1,\pi_*(\mathscr O_{\mathbb P(E)}(1)))=\bigoplus_{i=1}^d H^0(\mathbb P^1,\mathscr O_{\mathbb P^1}(n_i)).$$
Notice that  a basis of the $k$-vector space $H^0(\mathbb P^1,\mathscr O_{\mathbb P^1}(n_i))=k[T_{i,0},T_{i,1}]_{n_i}$ is  $T_{i,0}^{n_i}$, $T_{i,0}^{n_i-1}T_{i,1}$, $\ldots$, $T_{i,0}T_{i,1}^{n_i-1}$, $T_{i,1}^{n_i}$. Here $\mathbb P^1=\Proj(k[T_{i,0},T_{i,1}])$, with $T_{i,0}$ and $T_{i,1}$ two independent variables over $k$, $i=1,\dots,d$, and  $k[T_{i,0},T_{i,1}]_{n_i}$ is the $k$-vector space of all homogeneous polynomials in $T_{i,0}$ and $T_{i,1}$ of degree $n_i$.

The homogeneous ideal $\wp:=\mathscr I_+(S_{n_1,\dots,n_d})$ of $S_{n_1,\dots,n_d}$ in $\mathbb P^N$ (generated by all homogeneous polynomials vanishing on 
$S_{n_1,\dots,n_d}$) is thus the ideal generated by the $2\times 2$ minors of the matrix $A$ in the polynomial ring $k[X_{1,0},\dots,X_{1,n_1},\dots,X_{d,0},\dots, X_{d,n_d}].$   

We want to exhibit $N-2=\sum_{i=1}^dn_i+d-3$ homogeneous equations defining  $S_{n_1,\dots,n_d}$ in $\mathbb P^N$ set-theoretically.
In order to do it, the first step is to introduce a class of polynomials, which we call bridges and which will be crucial in order to detect the equations defining the rational normal scrolls. The bridges are defined in the following way. 

Let $a$ and $b$ be positive integers and 
let $m$ be the  least common multiple of $a$ and $b.$ We can write $m=ap=bq$ and for every $\alpha=0,\ldots,m$ we can divide $\alpha$ by $p$ and by $q$, thus getting $$\alpha=cp+r=eq+f,$$ where $0\le r \le p-1$ and $0\le f \le  q-1.$

In the polynomial ring $k[X_0,\dots,X_a,Y_0,\dots,Y_b]$ we consider the  polynomial
\begin{equation}\label{brdge1}B_{a,b}(X_0,\ldots,X_a,Y_0,\ldots,Y_b):= \sum_{\alpha=0}^m(-1)^{\alpha} {m \choose \alpha} X_{a-c}^{p-r}X_{a-c-1}^rY_e^{q-f}Y_{e+1}^f.\end{equation} 
We notice that if $\alpha=m$ then $c=a,$ $r=0$,  $e=b$ and $f=0$; in this case we let  $X_{-1}=1$ and $Y_{b+1}=1$.  
The polynomial $$B_{a,b}(X,Y):=B_{a,b}(X_0,\ldots,X_a,Y_0,\ldots,Y_b)$$ is called the  {\em bridge between $k[X_0,\dots,X_a]$ and $\;k[Y_0,\dots, Y_b]$} and it is homogeneous of  degree $m/a+m/b=p+q.$ When it is clear from the context, we shall simply write $B_{a,b}$ instead of $B_{a,b}(X,Y).$  The bridge $B_{a,b}$  has the following two relevant properties.

\vskip 3mm \noindent $\bullet $ {\em Property $1$.} For every $u,s,t,v\in k$ we have $$B_{a,b}(us^a,us^{a-1}t,\dots,ust^{a-1},ut^a,vs^b,vs^{b-1}t,\dots,vst^{b-1},vt^b)=0.$$
Indeed, what we have to do is to replace in $B_{a,b}$  every   $X_j$ by $us^{a-j}t^j$ and every $Y_h$ by $vs^{b-h}t^h.$ We get:
$$\sum_{\alpha=0}^m(-1)^{\alpha}{m\choose \alpha} (us^ct^{a-c})^{p-r}(us^{c+1}t^{a-c-1})^r(vs^{b-e}t^e)^{q-f}(vs^{b-e-1}t^{e+1})^f=$$ 
$$\sum_{\alpha=0}^m(-1)^{\alpha}{m\choose\alpha} u^{p}s^{c(p-r)+r(c+1)+(q-f)(b-e)+f(b-e-1)} v^{q}t^{(a-c)(p-r)+r(a-c-1)+e(q-f)+f(e+1)}=$$  $$=u^pv^qs^mt^m\left(\sum_{\alpha=0}^m(-1)^{\alpha}{m\choose\alpha}\right)=0.$$

\vskip 3mm \noindent$\bullet$ {\em Property $2$.} For every $s,t,z,w\in k$ we have 
$$B_{a,b}(s^a,s^{a-1}t,\dots,st^{a-1},t^a,z^b, z^{b-1}w,\dots,zw^{b-1},w^b)=(tz-sw)^m.$$

This time we have to replace  in $B_{a,b}$ every $X_j$ by $s^{a-j}t^j$ and every $Y_h$ by $z^{b-h}w^h.$ We get:
$$\sum_{\alpha=0}^m(-1)^{\alpha}{m\choose\alpha}(s^ct^{a-c})^{p-r}(s^{c+1}t^{a-c-1})^r(z^{b-e}w^e)^{q-f}(z^{b-e-1}w^{e+1})^f=$$ 
$$=\sum_{\alpha=0}^m(-1)^{\alpha}{m\choose \alpha}s^{c(p-r)+r(c+1)}t^{(p-r)(a-c)+r(a-c-1)}z^{(q-f)(b-e)+f(b-e-1)}w^{e(q-f)+f(e+1)} =$$ 
$$=\sum_{\alpha=0}^m(-1)^{\alpha}{m\choose\alpha}s^{\alpha}w^{\alpha}t^{m-{\alpha}}z^{m-{\alpha}}=(tz-sw)^m.$$

We notice that 
$$\{us^a,us^{a-1}t,\dots,ust^{a-1},ut^a,vs^b,vs^{b-1}t,\dots,vst^{b-1},vt^b\}$$ 
are the parametric equations of the rational normal scroll $S_{a,b}$ defined by the vanishing of the $(2\times 2)$-minors of the matrix 
$$M_{a,b}=\begin{pmatrix}
 X_0     &  X_1& \dots & X_{a-1} & Y_0 & Y_1 & \dots & Y_{b-1}\\ 
   X_1   &  X_2& \dots &X_a & Y_1 & Y_2 & \dots & Y_b
\end{pmatrix}.$$ 
Hence Property 1 implies that {\em $B_{a,b}$ is in the ideal generated by the $(2 \times 2)$-minors of $M_{a,b}.$}  Namely, $S_{a,b}=\overline{W}$ where $W$ is the set of points with coordinates (see e.g. \cite{Bel})
$$W=[us^a: us^{a-1}t: \dots :ust^{a-1}:ut^a:vs^b:vs^{b-1}t:\dots :vst^{b-1}:vt^b].$$
 By Property 1 we get  that $W$, and hence $S_{a,b}=\overline{W}$ is contained in the zero-locus of $B_{a,b}$. This implies that $B_{a,b}$ is contained in the defining ideal of $S_{a,b},$ which is the ideal generated by the $2\times 2$ minors of $M_{a,b}.$

Notice that, given $a$ and $b,$ while computing $B_{a,b}$ we can   avoid all the nasty euclidean divisions which appear in the definition itself. Better, one can do as follows. Let us consider  the following  list of monomials of degree $p$ in the $X_i$'s:  
$$\{{\bf X_a^p},X_{a}^{p-1}X_{a-1},\dots,X_{a}X_{a-1}^{p-1},{\bf X_{a-1}^p},X_{a-1}^{p-1}X_{a-2},\dots, X_{a-1}X_{a-2}^{p-1},{\bf X_{a-2}^p},\dots,X_1X_0^{p-1},{\bf X_0^p}\}.$$ 
In the same way one can write down the following list of monomials of degree $q$ in the $Y_i$'s:
$$\{{\bf Y_0^q},Y_0^{q-1}Y_1,\dots, Y_0Y_1^{q-1},{\bf Y_1^q},Y_1^{q-1}Y_2,\dots, Y_1Y_2^{q-1},
{\bf Y_2^q},\dots, Y_{b-1}Y_b^{q-1},{\bf Y_b^q}\}.$$ 
The first list has $ap+1=m+1$ terms and the second, $bq+1=m+1$ terms. The bridge $B_{a,b}$ is the sum of the products of the corresponding monomials in the two lists with appropriate binomial coefficients.

\begin{examps}\label{scr1}{\em
\begin{enumerate}
\item If $a=2,$ $b=4$, then $m=4,$ $p=2,$ and $q=1.$  The two lists are the following
$\{X_2^2,X_2X_1,X_1^2,X_1X_0,X_0^2\}$ and
$\{Y_0,Y_1,Y_2,Y_3,Y_4\}$. 

Hence $B_{2,4}(X,Y)=X_2^2Y_0-{4 \choose 1}X_2X_1Y_1+{4 \choose 2}X_1^2Y_2-{4 \choose 3}X_1X_0Y_3+X_0^2Y_4.$

\item If $a=2,$ $b=3$, then $m=6,$ $p=3$ and $q=2.$  The two lists are the following
$\{X_2^3,X_2^2X_1,X_2X_1^2,X_1^3,X_1^2X_0,X_1X_0^2,X_0^3\}$ and
$\{Y_0^2,Y_0Y_1,Y_1^2,Y_1Y_2,Y_2^2,Y_2Y_3,Y_3^2\}.$ Hence 
$B_{2,3}(X,Y)=X_2^3Y_0^2 -{6\choose 1}X_2^2X_1Y_0Y_1+{6 \choose 2}   X_2X_1^2Y_1^2-{6 \choose 3} X_1^3Y_1Y_2+{6 \choose 4}X_1^2X_0Y_2^2-{6 \choose 5}X_1X_0^2Y_2Y_3+X_0^3Y_3^2.$

\item If $a=b$, then $m=a$, $p=q=1.$ The two lists are $\{X_a,X_{a-1},X_{a-2},\dots, X_2,X_1,X_0\}$  and $\{Y_0,Y_1,Y_2,\dots,Y_{a-2},Y_{a-1},Y_a\}$ Hence $B_{a,a}(X,Y)=\sum_{j=0}^a(-1)^j{a\choose j}X_{a-j}Y_j.$

\item  If $a=3,$ $b=4$, then $m=12,$ $p=4$ and $q=3.$  Then we get 

$B_{3,4}(X,Y)=X_3^4Y_0^3-12X_3^3X_2Y_0^2Y_1+{12 \choose 2}X_3^2X_2^2Y_0Y_1^2-{12 \choose 3}X_3X_2^3Y_0Y_1^3+{12 \choose 4}X_2^4Y_1^2Y_2
-{12 \choose 5}X_2^3X_1Y_1Y_2^2+{12 \choose 6}X_2^2X_1^2Y_2^3-{12 \choose 7}X_2X_1^3Y_2^2Y_3+{12 \choose 8}X_1^4Y_2Y_3^2-{12 \choose 9}X_1^3X_0Y_3^3+
{12 \choose 10}X_1^2X_0^2Y_3^2Y_4-12X_1X_0^3Y_3Y_4^2+X_0^4Y_4^3.$
\end{enumerate}}\end{examps}
We are now ready to prove the main result of this section.

\medskip

{\em Proof of Theorem $\ref{mainscroll}$.} We divide the proof in two steps.

\medskip

{\em Step $1$.} {\em The following inequality holds:} 
$$\ara(S_{n_1,\ldots,n_d})\leq N-2=\sum_{i=1}^dn_i+d-3.$$

To prove Step 1 it is enough to find $\sum_{i=1}^dn_i+d-3$ homogeneous polynomials defining $S_{n_1,\ldots,n_d}$ in $\mathbb P^N$ set-theoretically.

Let us consider the polynomials $\{ F_{1,1},\dots,F_{1,n_1-1}\}$ in $k[X_{1,0},\dots, X_{1,n_1}]$ whose corresponding equations define set-theoretically the rational normal curve $C_{n_1}$ 
in $\mathbb{P}^{n_1}$, see (\ref{set}). Similarly we consider the polynomials $\{F_{2,1},\dots, F_{2,n_2-1}\}$ and so on up to $\{F_{d,1},\dots, F_{d,n_d-1}\}.$ This is a collection of $\sum_{i=1}^dn_i-d$ polynomials in $k[X_{1,0},\dots, X_{1,n_1},\dots,X_{d,0},\dots, X_{d,n_d}]$ belonging to the homogeneous  ideal $\wp=\mathscr I_+(S_{n_1,\dots,n_d})$.

We are going  to find some $2d-3$ more equations. This will be achieved by considering the  bridges  $B_{n_i,n_j}$ between $k[X_{i,0},\dots,X_{i,n_i}]$ and $k[X_{j,0},\dots,X_{j,n_j}]$ for every $1\le i<j\le d$. If $m_{i,j}=n_ip_{i,j}=n_jq_{i,j}$ is the least common multiple of $n_i$ and $n_j,$  then $B_{n_i,n_j}$ is homogeneous of degree $p_{i,j}+q_{i,j}.$ 

By Property 1 of the bridges we have that for every $1\le i<j\le d$ the polynomial $B_{n_i,n_j}$ belongs to the ideal of 
the polynomial ring $k[X_{i,0},X_{i,1},\dots, X_{i,n_i}, X_{j,0},X_{j,1},\dots, X_{j,n_j}]$ generated by the 
$(2\times 2)$-minors of the matrix
$$\begin{pmatrix}
X_{i,0}&X_{i,1}&\ldots&X_{i,n_i-1}&X_{j,0}&X_{j,1}&\ldots&X_{j,n_j-1}\\
X_{i,1}&X_{i,2}&\ldots&X_{i,n_i}&X_{j,1}&X_{j,2}&\ldots&X_{j,n_j}\\
\end{pmatrix}.$$ 
In particular, $B_{n_i,n_j}\in\wp=\mathscr I_+(S_{n_1,\dots,n_d})$. 

We associate a weight to the bridges by letting {\em weight $(B_{n_i,n_j}):=i+j.$} Hence we have that $B_{n_1,n_2}$ has weight 3, $B_{n_1,n_3}$ has weight 4, $B_{n_1,n_4}$ and $B_{n_2,n_3}$ have weight 5, $B_{n_1,n_5}$ and $B_{n_2,n_4}$ have weight 6, $B_{n_1,n_6},B_{n_2,n_5}$ and $B_{n_3,n_4}$ have weight 7 and so on.
Notice that the possible weight for a bridge is an integer $w$ such that  $3\le w \le  2d-1.$ 
Now for every $k=3,\dots , 2d-1,$ let $r_k$  be the least common multiple of the numbers  $p_{i,j}+q_{i,j}$ when $i+j=k,$ i.e. 
$$r_k:=\lcm \{p_{i,j}+q_{i,j}\ | \ i+j=k\}.$$ 
Further for every $i$ and $j$ such that $i+j=k$ we let 
$$c_{i,j}:=\frac{r_k}{p_{i,j}+q_{i,j}}.$$ 
Finally for every $k=3,\dots,2d-1,$ we let 
$$G_k:=\sum_{i+j=k} B_{n_i,n_j}^{c_{i,j}}.$$ 
It is clear that  $G_k$ is an  homogeneous polynomial of degree $r_k$ for every $k=3,\dots,2d-1.$ The polynomials $G_k$ are in $\wp$ because we have already seen that the bridges are in $\wp.$ 

For example we have $r_3=p_{1,2}+q_{1,2}$ so that $c_{1,2}=1$ and $G_3=B_{n_1,n_2}.$ Also $r_4=p_{1,3}+q_{1,3}$ so that $c_{1,3}=1$ and $G_4=B_{n_1,n_3}.$  Instead we have $r_5=\lcm(p_{1,4}+q_{1,4},p_{2,3}+q_{2,3})$,
so that $$c_{1,4}=\frac{r_5}{p_{1,4}+q_{1,4}}, \ \ c_{2,3}=\frac{r_5}{p_{2,3}+q_{2,3}},\;\;\text{and} \;\;
G_5=B_{n_1,n_4}^{c_{1,4}}+B_{n_2,n_3}^{c_{2,3}}.$$

Set 
$$J=(F_{1,1},\dots,F_{1,n_1-1},\dots,F_{d,1},\dots, F_{d,n_d-1},G_3,\dots,G_{2d-1}).$$ 
We are going to prove that the equations corresponding to these  $\sum n_i-d+2d-3=\sum n_i+d-3$ homogeneous polynomial 
  define set-theoretically the scroll $S_{n_1,\dots,n_d}$.

In other words, it's enough to prove the following

\begin{equation}\label{radical}  \wp=\sqrt{J}.\end{equation}

Clearly, $J\subseteq \wp$, so that $\root\of J\subseteq\wp$. On the other hand, by Nullstellensatz, the reverse inclusion is equivalent with $\mathscr V_+(J)\subseteq \mathscr V_+(\wp)$. To prove this latter inclusion, let  $P$ be an arbitrary point of $\mathscr V_+(J)$. We have  to show that  $P\in \mathscr V_+(\wp)$.
Since 
$P\in \mathscr V_+(F_{1,1},\dots,F_{1,n_1-1},\dots,F_{d,1},\dots, F_{d,n_d-1}),$ the coordinates of $P$ are of the following form
$$[t_1^{n_1},t_1^{n_1-1}u_1,\dots,t_1u_1^{n_1-1},u_1^{n_1};\ldots;t_d^{n_d},t_d^{n_d-1}u_d,\dots,t_du_d^{n_d-1},u_d^{n_d}],$$ 
or, in a compact way, 
$$\{X_{i,j}=t_i^{n_i-j}u_i^j\}_{\substack{i=1,\dots,d,\; j=0,\dots,n_i}}.$$ 
Let us consider the matrix $$D:=\begin{pmatrix}
  t_1    &   t_2 & \dots & t_d \\
   u_1   &  u_2 & \dots & u_d
\end{pmatrix}$$ 
and for every $1\le i < j \le d,$ let $\alpha_{i,j}$ be the $2\times 2$ minor involving its i-th and j-th column. We have $0=G_3(P)=B_{n_1,n_2}(P),$ hence, by Property 2 of the bridges,   
$$0=B_{n_1,n_2}(t_1^{n_1},t_1^{n_1-1}u_1,\dots, t_1u_1^{n_1-1},u_1^{n_1},t_2^{n_2},t_2^{n_2-1}u_2,\dots, t_2u_2^{n_2-1},u_2^{n_2})=(u_1t_2-t_1u_2)^{m_{1,2}}.$$ 
This implies $\alpha_{1,2}=0.$ 

In the same way we have $0=G_4(P)=B_{n_1,n_3}(P),$ hence 
$$0=B_{n_1,n_3}(t_1^{n_1},t_1^{n_1-1}u_1,\dots, t_1u_1^{n_1-1},u_1^{n_1},t_3^{n_3},t_3^{n_3-1}u_3,\dots, t_3u_3^{n_3-1},u_3^{n_3})=(u_1t_3-t_1u_3)^{m_{1,3}}.$$ 
This implies $\alpha_{1,3}=0$. 

Further $0=G_5(P)=(B_{n_1,n_4}^{c_{1,4}}+B_{n_2,n_3}^{c_{2,3}})(P),$ hence
$$0=(B_{n_1,n_4}^{c_{1,4}}+B_{n_2,n_3}^{c_{2,3}})(P)=(B_{n_1,n_4}(P))^{c_{1,4}}+(B_{n_2,n_3}(P))^{c_{2,3}}=$$
$$=(B_{n_1,n_4}(t_1^{n_1},t_1^{n_1-1}u_1,\dots, t_1u_1^{n_1-1},u_1^{n_1},t_4^{n_4},t_4^{n_4-1}u_4,\dots, t_4u_4^{n_4-1},u_4^{n_4}))^{c_{1,4}}+$$
$$+(B_{n_2,n_3}(t_2^{n_2},t_2^{n_2-1}u_2,\dots, t_2u_2^{n_2-1},u_2^{n_2},t_3^{n_3},t_3^{n_3-1}u_3,\dots, t_3u_3^{n_3-1},u_3^{n_3}))^{c_{2,3}}=$$  $$=(u_1t_4-t_1u_4)^{m_{1,4}c_{1,4}}+(u_2t_3-t_2u_3)^{m_{2,3}c_{2,3}}.$$ 
This implies $$\alpha_{1,4}^{m_{1,4}c_{1,4}}+\alpha_{2,3}^{m_{2,3}c_{2,3}}=0.$$
In the same way, for every $k=3,\dots,2d-1$, we get  
$$0=\sum_{i+j=k}\alpha_{i,j}^{m_{i,j}c_{i,j}}=\sum_{i+j=k}\alpha_{i,j}^{e_{i,j}},$$ 
where, for simplicity, we put $e_{i,j}:=m_{i,j}c_{i,j}.$

\vskip 3mm We claim that this implies $\alpha_{i,j}=0$ for every $1\le i<j\le d.$ 

\vskip 3mm  To prove this claim we order the $\alpha_{i,j}$'s as follows: $$\alpha_{i,j}<\alpha_{h,k}\iff \begin{cases}
 i+j<h+k,\;  \text{ or} \\

   i+j=h+k\; \;\text{and}  \ \   i<h.
\end{cases}$$  First observe that 
$\alpha_{1,2}=\alpha_{1,3}=0,$ 
so that we can argue by induction. Let us assume that $\alpha_{1,3}<\alpha_{a,b}$ and that $\alpha_{h,k}=0$ for every $\alpha_{h,k}< \alpha_{a,b}.$ One has 
$$\alpha_{a,b}^{e_{a,b}+1}=\alpha_{a,b}\left (\sum_{i+j=a+b}\alpha_{i,j}^{e_{i,j}}\right)-\alpha_{a,b}\left(\sum_{\substack{i+j=a+b\\ (i,j)\neq (a,b)}}\alpha_{i,j}^{e_{i,j}}\right)
$$
We only need to prove that if $(i,j)\neq (a,b)$ and $i+j=a+b,$ then $\alpha_{i,j}\alpha_{a,b}=0.$ If $i<a,$ then $\alpha_{i,j}<\alpha_{a,b}$ so that, by the inductive assumption $\alpha_{i,j}=0$ and we are done. If, instead, $i>a,$ then $j<b$ so that  
$$a<i<j<b.$$ 
By Pl\"ucker's relations, we have 
$$\alpha_{i,j}\alpha_{a,b}-\alpha_{a,j}\alpha_{i,b}+\alpha_{a,i}\alpha_{j,b}=0.$$ 
Since $\alpha_{a,i}<\alpha_{a,b}$ and $\alpha_{a,j}<\alpha_{a,b},$ we have $\alpha_{a,i}=\alpha_{a,j}=0$ and  the claim is proved.

As a consequence we get that the matrix $D$ has rank one. But this clearly implies that the matrix

$$\begin{pmatrix}
  t_1^{n_1}    & t_1^{n_1-1} u_1& \dots & t_1u_1^{n_1-1} & \dots   & t_d^{n_d} & t_d^{n_d-1}u_d & \dots & t_du_d^{n_d-1} \\
  t_1^{n_1-1} u_1 & t_1^{n_1-2} u_1^2 & \dots & u_1^{n_1} & \dots 
      &  t_d^{n_d-1} u_d & t_d^{n_d-2} u_d^2 & \dots & u_d^{n_d}
\end{pmatrix}$$ has also rank one.  This  means that the point $P$ is in $\mathscr V_+(\wp)$, which proves Step 1.

\medskip

{\em Step $2$.} {\em The following inequality holds:
\begin{equation}\label{top}\ara(S_{n_1,\ldots,n_d})\geq N-2=\sum_{i=1}^dn_i+d-3.\end{equation}}

The proof of this step is topological. We first notice the fact that the cokernel of the canonical restriction map $\alpha\colon \Pic(\mathbb P^N)\to\Pic(S_{n_1,\ldots,n_d})$  is isomorphic to $\mathbb Z$ (and this holds in arbitrary characteristic).
 
Now, assume by way of contradiction that $\ara(S_{n_1,\ldots,n_d})\leq N-3$.  If the characteristic $p$ of the ground field $k$  is $0$,
then by Lefschetz's principle we can assume $k=\mathbb C$. 
Since
$\ara(S_{n_1,\ldots,n_d})\leq N-3$, then  by Corollary \ref{laz1} the map $\alpha$ imust be an isomorphism. This is a contradiction because $\Pic(S_{n_1,\ldots,n_d})\cong\mathbb Z\times\mathbb Z$.

If instead $p>0$, the inequality $\ara(Y)\leq N-3$, together with  Theorem \ref{etale}, imply that $\Coker(\alpha)$ is a finite $p$-group, which is again a contradiction. This proves Step 2 in arbitrary characteristic. 

Notice that in Step 2 there is nothing to prove if $\dim(S_{n_1,\ldots,n_d})=2$ because in this case
$\codim_{\mathbb P^N}(S_{n_1,\ldots,n_d})=N-2$. If instead $\dim(S_{n_1,\ldots,n_d})=3$ then $\codim_{\mathbb P^N}(S_{n_1,\ldots,n_d})=N-3$, and since $\Pic(S_{n_1,\ldots,n_d})\cong \mathbb Z\times\mathbb Z$, Step 2 is also a consequence of Theorem \ref{main}, vi).

Then Step 1 and Step 2 conclude the proof of Theorem \ref{mainscroll}. \qed

\medskip\medskip

In characteristic zero Corollary \ref{laz1} and the proof of Step 2  yield actually the following result:

\begin{corollary}\label{scrolln} Under the hypotheses of Theorem $\ref{mainscroll}$, assume that the characteristic of $k$ is zero and $d\geq 3$. Then there exists no ample vector bundle $F$  of rank $\leq N-3$ on $\mathbb P^N$ and a global section  of $F$ vanishing precisely on $S_{n_1,\ldots,n_d}$.
Moreover, this upper bound for the rank of $F$ is optimal. \end{corollary}

\begin{rem*} Corollary \ref{scrolln} generalizes the following result noticed by Lazarsfeld in \cite{La}: the image $S_{1,1,1}$ of the Segre embedding $\mathbb P^2\times\mathbb P^1\hookrightarrow\mathbb P^5$ cannot be the zero locus of a global section of an ample vector bundle of rank two on $\mathbb P^5$.
\end{rem*}

\begin{examps}\label{final}{\em
1.  Let us consider the $2$-dimensional scroll $S_{n_1,n_2}$ in $\mathbb P^{n_1+n_2+1}$;  its  defining ideal $\wp=\mathscr I_+(S_{n_1,n_2})$ is  generated by the $2\times 2$ minors of the matrix
$$\begin{pmatrix}
 X_0 &  X_1 &\cdots  &X_{n_1-1} & Y_0 & Y_1 &\cdots & Y_{n_2-1}  \\
 X_1    &  X_2 & \cdots&X_{n_1} &Y_1 & Y_2 & \cdots & Y_{n_2}
\end{pmatrix}$$ 
The proof of Theorem \ref{mainscroll} shows in particular that $S_{n_1,n_2}$ is set-theoretic complete intersection in $\mathbb P^{n_1+n_2+1}$ via the following $n_1+n_2-1$ equations:
$$ F_{1,i}(X_0,\ldots,X_{n_1})=\sum_{\alpha=0}^i (-1)^{\alpha}{i\choose \alpha}X_{i+1}^{i-\alpha}X_{\alpha}X_i^{\alpha},\;\;i=1,\dots,n_1-1,$$
$$ F_{2,j}(Y_0,\ldots,Y_{n_2})=\sum_{\beta=0}^j (-1)^{\beta}{j\choose \beta}Y_{j+1}^{j-\beta}Y_{\beta}Y_j^{\beta},\;\;j=1,\dots,n_2-1,$$
and the bridge 
$B_{n_1,n_2}(X_0,\ldots,X_{n_1};Y_0,\ldots,Y_{n_2})$ (see the formula \eqref{brdge1}). These equations are simpler and of lower degree than the  equations found by Verdi in \cite{V1}.

\medskip  

2. In order to give the idea of the size of the polynomials involved in our computation, we now explicitely write down the equations defining set-theoretically the scroll $S_{2,2,3,4}$ in $\mathbb P^{14}.$

The defining ideal of this scroll is the ideal $\wp$ generated by the $2\times 2$ minors of the matrix

$$\setcounter{MaxMatrixCols}{20}
\begin{pmatrix}
    X_0  & X_1 & Y_0 & Y_1 & Z_0 & Z_1 & Z_2 & T_0 & T_1 & T_2 & T_3 \\
    X_1  & X_2 & Y_1 & Y_2 & Z_1 & Z_2 & Z_3 & T_1 & T_2  & T_3 & T_4
   \end{pmatrix}$$  This scroll has dimension 4 and codimension  9. The main result of this section proves that the arithmetic rank  is $ara(S_{2,2,3,4})=12.$ Namely $\wp$ is the radical of the ideal generated by the following polynomials.
   $$X_0X_2-X_1^2, \ \ Y_0Y_2-Y_1^2,\ \ Z_0Z_2-Z_1^2, \ \ Z_0Z_3^2-2Z_1Z_2Z_3+Z_2^3 $$ 
   $$T_0T_2-T_1^2, \ \ T_0T_3^2-2T_1T_2T_3+T_2^3, \ \  T_0T_4^3-3T_1T_3T_4^2+3T_2T_3^2T_4-T_3^4$$  
   $$B_{2,2}(X,Y), \ \  B_{2,3}(X,Z), \ \  B_{2,4}(X,T)^5+B_{2,3}(Y,Z)^3, \ \ B_{2,4}(Y,T), \ \ B_{3,4}(Z,T).$$
}\end{examps}

\begin{corollary}\label{scrolln1} Let $i\colon\mathbb P^{d-1}\times\mathbb P^1\hookrightarrow\mathbb P^{d(r+1)-1}$ be the Segre-Veronese embedding given by the complete linear system $|\mathscr O_{\mathbb P^{d-1}\times\mathbb P^1}(1,r)|$. 
Then the subvariety  $i(\mathbb P^{d-1}\times\mathbb P^1)$ is the set-theoretic intersection of $d(r+1)-3$ homogeneous equations in $\mathbb P^{d(r+1)-1}$.
\end{corollary}

\proof This is just Theorem \ref{mainscroll} applied to $S_{r,r,\ldots,r}=i(\mathbb P^{d-1}\times\mathbb P^1)$. \qed

\begin{rem*} Using ad-hoc methods and assuming that the characteristic of $k$ is $\neq 2$, Varbaro proved Corollary \ref{scrolln1} independently in the special case $r=2$ (see  \cite{Var}, Theorem 3.11).
\end{rem*}

\bigskip
\bigskip

\noindent{\begin{tabular}{l}   
Universit\`a degli Studi di Genova, Dipartimento di Matematica\\
Via Dodecaneso 35, I-16146 Genova, Italy\\ 
E-mails: badescu@dima.unige.it, valla@dima.unige.it
 \end{tabular}}

\end{document}